\documentclass[12pt]{amsart}
\usepackage[utf8]{inputenc}
\usepackage[T1]{fontenc}
\usepackage{amsfonts}
\usepackage[leqno]{amsmath}
\usepackage{amssymb, xcolor, comment, float}
\usepackage[foot]{amsaddr}
\usepackage[shortlabels]{enumitem}
\usepackage{mathdots}
\usepackage[a4paper, footskip=0.5in, headheight = 0.5in, top=1.25in, bottom=1.25in,  right=1in,  left=1in]{geometry}
\usepackage{hyperref}
\hypersetup{
colorlinks,
linkcolor=black,
urlcolor=darkblue,
citecolor=black,}
\usepackage[center]{caption}
\usepackage{eufrak}
\usepackage{marvosym}     
\usepackage{latexsym}
\usepackage[nodayofweek]{datetime}
\usepackage{tikz}
\usepackage{subfig}

\newtheorem{theorem}{Theorem}[section]
\newtheorem{lemma}[theorem]{Lemma}

\newtheorem{question}[theorem]{Question}

\newcommand{\poi}{\mathbb{N}} 

\newcounter{tbox}
\newcommand{\sta}[1]{\medskip\medskip\refstepcounter{tbox}\noindent{\parbox{\textwidth}{(\thetbox) \emph{#1}}}\vspace*{0.3cm}}
\newcommand{\mylongtitle}[1]{%
  \ifodd\value{page}%
    \protect\parbox{0.97\linewidth}{#1}\hfill%
  \else%
    \hfill\protect\parbox{0.97\linewidth}{#1}%
  \fi%
}

\title[Induced subdivisions with pinned branch vertices]{Induced subdivisions with pinned branch vertices}

\author{Sepehr Hajebi $^{\mathsection \parallel}$}

\address{$^{\mathsection}$Department of Combinatorics and Optimization, University of Waterloo, Waterloo, Ontario, Canada}
\date {\today}
\address{$^{\parallel}$ We acknowledge the support of the Natural Sciences and Engineering Research Council of Canada (NSERC), [funding reference number RGPIN-2020-03912].
Cette recherche a \'et\'e financ\'ee par le Conseil de recherches en sciences naturelles et en g\'enie du Canada (CRSNG), [num\'ero de r\'ef\'erence RGPIN-2020-03912]. This project was funded in part by the Government of Ontario.}

\begin{document}
\maketitle
\begin{abstract}
We prove that for all   $r \in \poi\cup \{0\}$ and $s,t\in \poi$, there exists   $\Omega=\Omega(r,s,t)\in \poi$ with the following property. Let $G$ be a graph and let $H$ be a subgraph of $G$ isomorphic to a $(\leq r)$-subdivision of $K_{\Omega}$. Then either $G$ contains $K_t$ or $K_{t,t}$ as an induced subgraph, or there is an induced subgraph $J$ of $G$ isomorphic to a proper $(\leq r)$-subdivision of $K_s$ such that every branch vertex of $J$ is a branch vertex of $H$. This answers in the affirmative a question of Lozin and Razgon. In fact, we show that both the branch vertices and the paths corresponding to the subdivided edges between them can be preserved. 

\end{abstract}

\section{Introduction}\label{sec:intro}

The set of all positive integers is denoted by $\poi$. Graphs in this paper have finite vertex sets, no loops and no parallel edges. Let $G=(V(G),E(G))$ be a graph. For $X \subseteq V(G)$, we denote by $G[X]$ the subgraph of $G$ induced by $X$ (we refer to induced subgraphs and their vertex sets interchangeably).

A \textit{subdivision} of a graph $H$ is a graph $H'$ obtained from $H$ by replacing the edges of $H$ with
pairwise internally disjoint induced paths of non-zero lengths between the corresponding ends (the \textit{length} of a path is its number of edges). In this case, we refer to the vertices of $H$ as the \textit{branch vertices} of $H'$, and to the induced paths in $H'$ replacing the edges of $H$ as the \textit{subdivision paths}. For   $p \in \poi\cup \{0\}$, we say $H'$ is a \textit{$(\leq p)$-subdivision} of $H$ if all subdivision paths in $H'$ have length at most $p+1$. We also say $H'$ is a \textit{proper subdivision} of $H$ if all subdivision paths in $H'$ have length at least two.
\medskip

In \cite{lozin}, Lozin and Razgon proved the following.

\begin{theorem}[Lozin and Razgon, Theorem 3 in \cite{lozin}]\label{thm:lozin}
    For all positive integers $r$ and $p$, there is a positive integer $m = m(r,p)$ such that every graph $G$ containing a $(\leq p)$-subdivision of $K_m$ as a subgraph contains either $K_{p,p}$ as a subgraph or a proper $(\leq  p)$-subdivision of $K_{r,r}$ as an induced subgraph.
\end{theorem}
It is worth noting that Theorem~\ref{thm:lozin} (modulo the ``shortness'' of the subdivision of $K_{r,r}$ in the outcome) is an immediate corollary of the main result of \cite{dvorak}, asserting that for every $p\in \poi$ and every graph $H$, the class of all graphs with no subgraph isomorphic to $K_{p,p}$ and no induced subgraph isomorphic to a subdivision of $H$ has ``bounded expansion.'' Roughly speaking, this means the conclusion of Theorem~\ref{thm:lozin} remains true under the weaker assumption that $G$ contains, as a subgraph, a ``short'' subdivision of any graph with sufficiently large average degree (rather than a large enough complete graph). This result in turn extends (while critically building upon the main ideas from) a famous theorem of K\"{u}hn and Osthus \cite{KO}, the proof of which is fairly long and complicated.

The proof of Theorem~\ref{thm:lozin}, however, is quite short. Moreover, it actually yields a proper $(\leq  p)$-subdivision of $K_{r,r}$ as an induced subgraph of $G$ whose branch vertices are also branch vertices of the initial $(\leq p)$-subdivision of $K_m$ in $G$. Accordingly, Lozin and Razgon \cite{lozin} asked whether this can be strengthened to guarantee a proper subdivision of a complete graph (rather than a complete bipartite graph) with the same property. More exactly, they proposed the following:

\begin{question}[Lozin and Razgon, Problem 1 in \cite{lozin}]\label{q:lozin}
    Is it true that for all positive integers $r$ and $p$, there is a positive integer $m = m(r,p)$ such that every graph $G$ containing a $(\leq p)$-subdivision of $K_m$ as a subgraph contains either $K_{p,p}$ as a subgraph or a proper $(\leq p)$-subdivision of $K_r$ as an induced subgraph, whose branch vertices are also the branch vertices of the $K_m$?
\end{question}

We give an affirmative answer to Question~\ref{q:lozin}. Indeed, we show that both the branch vertices and the subdivision paths between them can be preserved. In order to state our main result precisely, we need a few definitions.

Let $X$ be a set and let $q \in \poi\cup \{0\}$. We denote by $2^{X}$ the set of all subsets of $X$ and by $\binom{X}{q}$ the set of all $q$-subsets of $X$. Given a graph $G$, by a \textit{path in $G$} we mean an induced subgraph of $G$ which is a path. If $L$ is a path in $G$, we call the vertices of degree less than two in $L$ the \textit{ends} of $L$, and refer to the set of all other vertices in $L$ as the \textit{interior} of $L$, denoted by $L^*$ (so a path of length at most one has empty interior). A \textit{stable set} in $G$ is a set of pairwise non-adjacent vertices in $G$. Two subsets $X,Y\subseteq V(G)$ are said to be \textit{anticomplete} if there is no edge in $G$ with an end in $X$ and an end in $Y$. If $X=\{x\}$, we also say \textit{$x$ is anticomplete to $Y$} to mean $\{x\}$ and $Y$ are anticomplete.
\medskip 

For   $w\in \poi$, by a \textit{$w$-web in $G$} we mean a pair $(W,\Lambda)$ where
\begin{enumerate}[(W1), leftmargin=15mm, rightmargin=7mm]
     \item\label{W1} $W$ is a $w$-subset of $V(G)$;
     \item\label{W2} $\Lambda:\binom{W}{2}\rightarrow 2^{V(G)}$ is a map such that $\Lambda(\{x,y\})=\Lambda_{x,y}$ is a path in $G$ with ends $x,y$ for every $2$-subset $\{x,y\}$ of $W$; and
     \item\label{W3} we have $\Lambda_{x,y}\cap \Lambda_{x',y'}=\{x,y\}\cap \{x',y'\}$ for all distinct $2$-subsets $\{x,y\},\{x',y'\}$ of $W$.
 \end{enumerate}
Also, for $r \in \poi\cup \{0\}$, an \textit{$(r,w)$-web in $G$} is a $w$-web $(W,\Lambda)$ in $G$ such that for every $2$-subset $\{x,y\}$ of $W$, the path $\Lambda_{x,y}$ has length at most $r+1$. It follows that for every $r \in \poi\cup \{0\}$, $G$ contains a $(\leq r)$-subdivision of $K_w$ as a subgraph if and only if there is an $(r,w)$-web in $G$.

Our main result in this note is the following:

\begin{theorem}\label{mainshort}
   For all   $r \in \poi\cup \{0\}$ and $s,t\in \poi$, there exists   $\Omega=\Omega(r,s,t)\in \poi$ with the following property. Let $G$ be a graph and let $(W,\Lambda)$ be an $(r,\Omega)$-web in $G$. Then one of the following holds.

     \begin{enumerate}[\rm (a)]
        \item\label{mainshort_a} $G$ contains an induced subgraph isomorphic to $K_t$ or $K_{t,t}$.
         \item\label{mainshort_b} There exists $S\subseteq W$ with $|S|=s$ such that:
        \begin{itemize}
        \item[-] $S$ is a stable set in $G$; 
         \item[-] for all three vertices $x,y,z\in S$, $x$ is anticomplete to $\Lambda^*_{y,z}$; and
            \item[-] for all distinct $2$-subsets $\{x,y\},\{x',y'\}$ of $S$, $\Lambda^*_{x,y}$ and $\Lambda^*_{x',y'}$ are anticomplete in $G$.
        \end{itemize}
      \end{enumerate}

\end{theorem}
We prove this theorem in the next section. One may readily observe that Theorem~\ref{mainshort} answers Question~\ref{q:lozin} in the affirmative. 

\section{Proof of Theorem~\ref{mainshort}}\label{sec:proof}

Let us begin with the multicolor version of Ramsey's Theorem for complete uniform hypergraphs:

\begin{theorem}[Ramsey \cite{multiramsey}]\label{multiramsey}
For all   $f,g,n\in \poi$, there exists   $\rho(f,g,n)\in \poi$ with the following property. Let $U$ be a set of cardinality at least $\rho(f,g,n)$ and let $F$ be a non-empty set of cardinality at most $f$. Let $\Phi:\binom{U}{g}\rightarrow F$ be a map. Then there exist $i\in F$ and $Z\subseteq U$ with $|Z|=n$ such that for every $X\in \binom{Z}{g}$, we have $\Phi(X)=i$.
\end{theorem}

The next two lemmas have similar proofs, both relying on Theorem~\ref{multiramsey}. For an integer $n$, we denote by $[n]$ the set of all positive integers less than or equal to $n$.

\begin{lemma}\label{lem:pinned}
    For all   $a,b,s\in \poi$, there exists   $\tau=\tau(a,b,s)\in \poi$ with the following property. Let $G$ be a graph and let $(W,\Lambda)$ be a $\tau$-web in $G$. Then one of the following holds.
    \begin{enumerate}[\rm (a)]
        \item\label{lem:pinned_a} There exists $A\subseteq W$ with $|A|=a$ and a collection $\mathcal{B}$ of pairwise disjoint $2$-subsets of $W\setminus A$ with $|\mathcal{B}|=b$ such that for every $x\in A$ and every $\{y,z\}\in \mathcal{B}$, $x$ has a neighbor in $\Lambda_{y,z}$.
        \item\label{lem:pinned_b} There exists $S\subseteq W$ with $|S|=s$ such that:
        \begin{itemize}
        \item[-] $S$ is a stable set in $G$; and
         \item[-] for all three vertices $x,y,z\in S$, $x$ is anticomplete to $\Lambda_{y,z}^*$.
        \end{itemize}
    \end{enumerate}
\end{lemma}
\begin{proof}
We claim that 
$$\tau(a,b,s)=\rho(8,3,\max\{3a+2b,s\})$$
satisfies the lemma, where $\rho(\cdot,\cdot,\cdot)$ comes from Theorem~\ref{multiramsey}.

Assume that \ref{lem:pinned}\ref{lem:pinned_a}
does not hold. Fix an enumeration $W=\{w_1,\ldots, w_{\tau}\}$ of $W$. For every $3$-subset $T$ of $[\tau]$, let $\Phi(T)\in 2^{[3]}$ be defined as follows. Let $T=\{t_1,t_2,t_3\}$ such that $t_1<t_2<t_3$. Then for every $i\in [3]$, we have $i\in \Phi(T)$ if and only if $w_{t_i}$ has a neighbor in $G$ in $\Lambda_{w_{t_j},w_{t_k}}$, where $\{j,k\}=[3]\setminus \{i\}$.

It follows that the map $\Phi:\binom{[\tau]}{3}\rightarrow 2^{[3]}$ is well-defined. Due to the choice of $\tau$, we can apply Theorem~\ref{multiramsey} and deduce that there exists $F\subseteq [3]$ as well as $Z\subseteq [\tau]$ with $|Z|=\max\{3a+2b,s\}$ such that for every $T\in \binom{Z}{3}$, we have $\Phi(T)=F$. In particular, since $|Z|\geq 3a+2b$, we may choose $I_1,I_2,I_3,J,K\subseteq Z$ with $|I_1|=|I_2|=|I_3|=a$ and $|J|=|K|=b$, such that
$$\max I_1<\min J\leq \max J <\min I_2\leq \max I_2<\min K\leq \max K<\min I_3.$$
It follows that $I_1,I_2,I_3,J,K$ are pairwise disjoint. We now show that:

\sta{\label{st:Fisedgeless} $F$ is empty.}

Suppose not. Pick $f\in F\subseteq [3]$. Write $J=\{j_1,\ldots, j_b\}$ and $K=\{k_1,\ldots, k_b\}$. Note that for every $i\in I_f$ and every $t\in [b]$, $\{i,j_t,k_t\}$ is a $3$-subset of $Z$. Thus, we have $\Phi(\{i,j_t,k_t\})=F$, and so $f\in \Phi(\{i,j_t,k_t\})$. This, along with the definition of $\Phi$ and the sets $I_1,I_2,I_3,J,K$ implies that for every $i\in I_f$ and every $t\in [b]$, $w_i$ has a neighbor in $\Lambda_{w_{j_t},w_{k_t}}$. But then $A=\{w_i:i\in I_f\}$ and $\mathcal{B}=\{(w_{j_t},w_{k_t}):t\in [b]\}$ satisfy \ref{lem:pinned}\ref{lem:pinned_a}, a contradiction. This proves \eqref{st:Fisedgeless}.

\medskip

Since $|Z|\geq \max\{s,3\}$, it follows from \eqref{st:Fisedgeless} that for every $Z'\subseteq Z$ with $|Z'|=s$,  $S=\{w_i:i\in Z\}\subseteq W$ satisfies \ref{lem:pinned}\ref{lem:pinned_b}. This completes the proof of Lemma~\ref{lem:pinned}.
\end{proof}

\begin{lemma}\label{lem:cleaninterior}
    For all   $c,s\in \poi$, there exists   $\sigma=\sigma(c,s)\in \poi$ with the following property. Let $G$ be a graph and let $(W,\Lambda)$ be a $\sigma$-web in $G$. Then one of the following holds.

     \begin{enumerate}[\rm (a)]
          \item\label{lem:cleaninterior_a} There are disjoint subsets $\mathcal{C}$ and $\mathcal{C}'$ of $\binom{W}{2}$ with $|\mathcal{C}|=|\mathcal{C}'|=c$ such that for every $\{x,y\}\in \mathcal{C}$ and every $\{x',y'\}\in \mathcal{C}'$, $\Lambda_{x,y}^*$ and $\Lambda_{x',y'}^*$ are not anticomplete in $G$. 
        \item\label{lem:cleaninterior_b} There exists $S\subseteq W$ with $|S|=s$ such that for all distinct $2$-subsets $\{x,y\},\{x',y'\}$ of $S$, $\Lambda^*_{x,y}$ and $\Lambda^*_{x',y'}$ are anticomplete in $G$.
      \end{enumerate}

\end{lemma}

\begin{proof}
We claim that 
$$\sigma(c,s)=\rho(2^{15},4,\max\{4c,s\})$$
satisfies the lemma, with $\rho(\cdot,\cdot,\cdot)$ as in Theorem~\ref{multiramsey}. Assume that \ref{lem:cleaninterior}\ref{lem:cleaninterior_a}
does not hold.

Fix an enumeration $W=\{w_1,\ldots, w_{\sigma}\}$ of $W$. For every $4$-subset $T=\{t_1,t_2,t_3,t_4\}$ of $[\sigma]$ with $t_1<t_2<t_3<t_4$, let $\Phi(T)$ be the set of all $2$-subsets $\{\{i,j\},\{i',j'\}\}$ of $\binom{[4]}{2}$ for which $\Lambda_{w_{i},w_j}^*$ and $\Lambda_{w_{i'},w_{j'}}^*$ are not anticomplete in $G$. Then the map
$$\Phi:\binom{[\sigma]}{4}\rightarrow 2^{\displaystyle\binom{ \binom{[4]}{2}}{2}}$$
is well-defined. Consequently, by the choice of $\sigma$ and Theorem~\ref{multiramsey}, there exists a collection $F$ of $2$-subsets of $\binom{[4]}{2}$ as well as a subset $Z$ of $[\sigma]$ with $|Z|=\max\{4c,s\}$ such that for every $T\in \binom{Z}{4}$, we have $\Phi(T)=F$. In particular, since $|Z|\geq 4c$, we may choose $T_1,T_2,T_3,T_4\subseteq Z$ with $|T_1|=|T_2|=|T_3|=|T_4|=c$ such that
$$\max T_1<\min T_2\leq \max T_2 <\min T_3\leq \max T_3<\min T_4.$$
It follows that $T_1,T_2,T_3,T_4$ are pairwise disjoint. Let us now show that:

\sta{\label{st:Fisedgeless2} $F$ is empty.}

Suppose not. Let $\{P,P'\}\in F$, where $P,P'$ are distinct $2$-subsets of $[4]$. Then we may write $P=\{i,j\}$ and $P'=\{i',j'\}$ such that $j\notin P'$ and $j'\notin P$ (while $i=i'$ is possible). Pick an element $t_i\in T_i$ and an element $t_{i'}\in T_{i'}$ (this is possible because $c\in \poi$). It follows that for every $t\in T_j$ and every $t'\in T_{j'}$, $\{t_i,t_{i'},t,t'\}$ is a $4$-subset of $Z$. Therefore, we have $\Phi(\{t_i,t_{i'},t,t'\})=F$, and so $\{P,P'\}\in \Phi(\{t_i,t_{i'},t,t'\})$. This, together with the definition of $\Phi$ and the sets $T_1,T_2,T_3,T_4$ implies that for every $t\in T_j$ and every $t'\in T_{j'}$, $\Lambda^*_{w_{t_i},w_t}$ and $\Lambda^*_{w_{t_{i'}},w_{t'}}$ are not anticomplete in $G$. But then $\mathcal{C}=\{(w_{t_{i}},w_{t}):t\in T_{j}\}$ and $\mathcal{C}'=\{(w_{t_{i'}},w_{t'}):t'\in T_{j'}\}$ satisfy \ref{lem:cleaninterior}\ref{lem:cleaninterior_a}, a contradiction. This proves \eqref{st:Fisedgeless2}.

\medskip

Since $|Z|\geq s$, there exists $Z'\subseteq Z$ with $|Z'|=s$. Now by \eqref{st:Fisedgeless2}, $S=\{w_i:i\in Z\}\subseteq W$ satisfies \ref{lem:cleaninterior}\ref{lem:cleaninterior_b}. This completes the proof of Lemma~\ref{lem:cleaninterior}.
\end{proof}

Combining Lemmas~\ref{lem:pinned} and \ref{lem:cleaninterior}, we deduce the following strengthening of Theorem~\ref{mainshort}:

\begin{theorem}\label{mainthm}
    For all   $a,b,c,s\in \poi$, there exists   $\theta=\theta(a,b,c,s)\in \poi$ with the following property. Let $G$ be a graph and let $(W,\Lambda)$ be a $\theta$-web in $G$. Then one of the following holds.

     \begin{enumerate}[\rm (a)]
         \item\label{thm:mainthm_a} There exists $A\subseteq W$ with $|A|=a$ and a collection $\mathcal{B}$ of pairwise disjoint $2$-subsets of $W\setminus A$ with $|\mathcal{B}|=b$ such that for every $x\in A$ and every $\{y,z\}\in \mathcal{B}$, $x$ has a neighbor in $\Lambda_{y,z}$.
          \item\label{thm:mainthm_b} There are disjoint subsets $\mathcal{C}$ and $\mathcal{C}'$ of $\binom{W}{2}$ with $|\mathcal{C}|=|\mathcal{C}'|=c$ such that for every $\{x,y\}\in \mathcal{C}$ and every $\{x',y'\}\in \mathcal{C}'$, $\Lambda_{x,y}^*$ is not anticomplete to $\Lambda_{x',y'}^*$ in $G$. 
        \item\label{thm:mainthm_c} There exists $S\subseteq W$ with $|S|=s$ such that:
        \begin{itemize}
        \item[-] $S$ is a stable set in $G$;
         \item[-] for all three vertices $x,y,z\in S$, $x$ is anticomplete to $\Lambda^*_{y,z}$; and
            \item[-] for all distinct $2$-subsets $\{x,y\},\{x',y'\}$ of $S$, $\Lambda^*_{x,y}$ is anticomplete to $\Lambda^*_{x',y'}$.
        \end{itemize}
      \end{enumerate}
\end{theorem}
\begin{proof}
Let $\theta=\theta(a,b,c,s)=\tau(a,b,\sigma(c,s))$, where $\tau(\cdot,\cdot,\cdot)$ and $\sigma(\cdot,\cdot)$ come from Lemmas~\ref{lem:pinned} and \ref{lem:cleaninterior}, respectively. Let $G$ be a graph and let $(W,\Lambda)$ be a $\theta$-web in $G$. The choice of $\theta$ calls for an application of Lemma~\ref{lem:pinned} to $(W,\Lambda)$. Since \ref{lem:pinned}\ref{lem:pinned_a} and \ref{mainthm}\ref{thm:mainthm_a} are identical, we may assume that \ref{lem:pinned}\ref{lem:pinned_b} holds. It follows that there exists $\Sigma\subseteq W$ with $|\Sigma|=\sigma(c,s)$ such that:
        \begin{itemize}
        \item $\Sigma$ is a stable set in $G$; and
         \item for all three vertices $x,y,z\in \Sigma$, $x$ is anticomplete to $\Lambda_{y,z}^*$.
         \end{itemize}
In particular, $(\Sigma,\Lambda|_{\binom{\Sigma}{2}})$ is a $\sigma(c,s)$-web in $G$ to which we can apply Lemma~\ref{lem:cleaninterior}.

The result now follows from the fact that \ref{mainthm}\ref{thm:mainthm_b} is identical to \ref{lem:cleaninterior}\ref{lem:cleaninterior_a} and \ref{mainthm}\ref{thm:mainthm_c} is identical to \ref{lem:cleaninterior}\ref{lem:cleaninterior_b} combined with the above two bullet conditions.
\end{proof}

In order to deduce Theorem~\ref{mainshort} from Theorem~\ref{mainthm}, we need the following well-know lemma. This has been observed several times independently, for instance, in Lemma 2 from \cite{lozin}. We give a proof, as it is short and keeps the note self-contained.

\begin{lemma}\label{lem:smallsetanti}
 For all   $r,t\in \poi$ there exists   $\xi=\xi(r,t)\in \poi$ with the following property. Let $G$ be a graph. Let $X_1,\ldots, X_{\xi}$ be pairwise disjoint subsets of $V(G)$, each of cardinality at most $r$. Assume that for all distinct $i,j\in [\xi]$, $X_i$ and $X_j$ are not anticomplete in $G$ (and so $X_1,\ldots, X_{\xi}$ are all non-empty). Then $G$ contains an induced subgraph isomorphic to $K_t$ or $K_{t,t}$.
\end{lemma}
\begin{proof}
Let $\xi(r,t)=\rho(2^{r^2},2,2t)$, with $\rho(\cdot,\cdot,\cdot)$ as in Theorem~\ref{multiramsey}. Suppose for a contradiction that $G$ contains neither $K_t$ nor $K_{t,t}$ as an induced subgraph. For each $i\in [\xi]$, fix an enumeration $X_i=\{x_{i,1},\ldots, x_{i,|X_i|}\}$.  For every $2$-subset $\{i,j\}$ of $[\xi]$, let 
$$\Phi(\{i,j\})=\{(f,f'):x_{i,f}x_{j,f'}\in E(G)\}\subseteq [|X_i|]\times [|X_j|]\subseteq [r]^2.$$
By the assumption of \ref{lem:smallsetanti}, the map $\Phi:\binom{[\xi]}{2}\rightarrow 2^{[r]^2}\setminus \{\emptyset\}$ is well-defined. From the choice of $\xi$ and Theorem~\ref{multiramsey}, it follows that there are disjoint $t$-subsets $Z,Z'$ of $[\xi]$ as well as a non-empty subset $F$ of $[r]^2$ such that for every $2$-subset $\{i,j\}$ of $Z\cup Z'$, we have $\Phi(\{i,j\})=F$. Moreover, the following holds.

\sta{\label{st:nondiag} For every $f\in [r]$, we have $(f,f)\notin F$.}

For otherwise it follows from the definition of $\Phi$ that $G[\{x_{i,f}:i\in Z\}]$ is isomorphic to $K_t$, a contradiction. This proves \eqref{st:nondiag}.

\medskip

From \eqref{st:nondiag} and the fact that $F\neq \emptyset$, it follows that for some distinct $f,f'\in [r]$, we have $(f,f')\in F$ while $(f,f)\in F$ and $(f,f)\in F$. By the definition of $\Phi$, for every $i\in Z$ and every $i'\in Z'$, we have $x_{i,f}x_{i',f'}\in E(G)$, while $\{x_{i,f}:i\in Z\}$ and $\{x_{i',f'}:i'\in Z'\}$ are both stable sets in $G$. But now $G[\{x_{i,f}:i\in Z\}\cup \{x_{i',f'}:i\in Z'\}]$ is isomorphic to $K_{t,t}$, a contradiction. This completes the proof of Lemma~\ref{lem:smallsetanti}
\end{proof}
We now prove our main result, which we restate:
\setcounter{section}{1}
\setcounter{theorem}{2}
\begin{theorem}\label{mainshort}
   For all   $r \in \poi\cup \{0\}$ and $s,t\in \poi$, there exists   $\Omega=\Omega(r,s,t)\in \poi$ with the following property. Let $G$ be a graph and let $(W,\Lambda)$ be an $(r,\Omega)$-web in $G$. Then one of the following holds.

     \begin{enumerate}[\rm (a)]
        \item\label{mainshort_a} $G$ contains an induced subgraph isomorphic to $K_t$ or $K_{t,t}$.
         \item\label{mainshort_b} There exists $S\subseteq W$ with $|S|=s$ such that:
        \begin{itemize}
        \item[-] $S$ is a stable set in $G$; 
         \item[-] for all three vertices $x,y,z\in S$, $x$ is anticomplete to $L^*_{y,z}$; and
            \item[-] for all distinct $2$-subsets $\{x,y\},\{x',y'\}$ of $S$, $\Lambda^*_{x,y}$ and $\Lambda^*_{x',y'}$ are anticomplete in $G$.
        \end{itemize}
      \end{enumerate}

\end{theorem}
\begin{proof}
Let $\xi=\xi(\max\{r+3,2r\},t)$ be as in Lemma~\ref{lem:smallsetanti}.
We prove that $\Omega=\Omega(r,s,t)=\theta(\xi,\xi,\xi,s)$ satisfies the theorem. Suppose that \ref{mainshort}\ref{mainshort_b} does not hold.

\sta{\label{st:pwtouch} There are pairwise disjoint subsets $X_1,\ldots, X_{\xi}$ of $V(G)$, each of cardinality at most $\max\{r+3,2r\}$, such that for all distinct $i,j\in [\xi]$, $X_i$ and $X_j$ are not anticomplete in $G$.}

To see this, note that by the choice of $\Omega$, we can apply Theorem~\ref{mainthm} to the $\Omega$-web $(W,\Lambda)$. Since \ref{mainshort}\ref{mainshort_b} and \ref{mainthm}\ref{thm:mainthm_c} are identical, we conclude that one of the following holds.
\begin{itemize}
    \item There exists $A=\{x_1,\ldots, x_{\xi}\}\subseteq W$ and a collection $\mathcal{B}=\{y_i,z_i\}:i\in [\xi]\}$ of pairwise disjoint $2$-subsets of $W\setminus A$ such that for all $i,j\in [\xi]$, $x_i$ has a neighbor in $\Lambda_{y_j,z_j}$; or
    \item There are two disjoint subsets of  $\binom{W}{2}$, namely $\mathcal{C}=\{\{x_i,y_i\}:i\in [\xi]\}$ and $\mathcal{C}'=\{\{x'_i,y'_i\}:i\in [\xi]\}$, such that for all $i,j\in [\xi]$, $\Lambda_{x_i,y_i}^*$ is not anticomplete to $\Lambda_{x'_j,y'_j}^*$ in $G$.  
\end{itemize}
For every $i\in [\xi]$, in the former case above, let $X_i=\{x_i\}\cup \Lambda_{y_i,z_i}$, and in the latter case above, let $X_i=\Lambda^*_{x_i,y_i}\cup \Lambda^*_{x'_i,y'_i}$. Then for all distinct $i,j\in [\xi]$, $X_i$ and $X_j$ are disjoint, and there is an edge in $G$ with an end in $X_i$ and an end in $X_j$. Also, recall that $(W,\Lambda)$ is an $(r,\Omega)$-web, and so $|X_i|\leq \max\{r+3,2r\}$ for all $i\in [\xi]$. This proves \eqref{st:pwtouch}.

\medskip

Now \eqref{st:pwtouch} along with the choice of $\xi$ and Lemma~\ref{lem:smallsetanti} implies that $G$ contains an induced subgraph isomorphic to $K_t$ or $K_{t,t}$. Hence, Theorem~\ref{mainshort}\ref{mainshort_a} holds.
\end{proof}
\setcounter{section}{2}
\section{Acknowledgement}
Our thanks to Bogdan Alecu, Maria Chudnovsky and Sophie Spirkl for stimulating discussions.

\end{document}